\documentclass[12pt]{amsart}
\usepackage{amsfonts}
\usepackage{latexsym,amsmath,amsthm,amssymb,amsfonts,mathrsfs}
\usepackage{graphicx}
\usepackage[usenames, dvipsnames]{color}
\usepackage{tikz}
\usepackage{pgfplots}
\pgfplotsset{width=10cm,compat=1.9}
\numberwithin{equation}{section}

\newtheorem{thm}{Theorem}[section]

\newtheorem{lem}{Lemma}[section]
\newtheorem{prop}{Proposition}[section]
\newtheorem{cor}{Corollary}[section]
\theoremstyle{definition}
\newtheorem{dfn}{Definition}[section]

\theoremstyle{remark}
\newtheorem{rem}{Remark}

\begin{document}

\title{Ancient solutions to the Curve Shortening Flow spanning the halfplane}

\author{John Man Shun Ma}
\address[]{}
\email{john.ma311@rutgers.edu}

\begin{abstract}
In this note we construct an infinite family of compact ancient solutions to the Curve Shortening Flow which span the halfplane. 
\end{abstract}

\date{\today}

\maketitle

\markboth{}{}

\section{Introduction}
The Curve Shortening Flow (CSF) is a family of immersed curves which moves along the curvature vector $\vec \kappa$. The CSF is the one dimensional example of the Mean Curvature Flow (MCF), the most studied extrinsic geometric flows. 

In this paper, we restrict our attention to ancient solutions to the CSF. A solution $\{ \gamma_t \}$ to the CSF is called ancient if $\gamma_t$ is defined in $t\in (-\infty, T_0)$ for some $T_0$. 

Besides the straight lines and the shrinking circles, the first non-trivial examples of ancient solutions are given by Abresch and Langer \cite{AL}, where they give a complete classification of compact self-shrinking solutions to the CSF. This result is generalized in \cite{Ha}, where Halldorsson gives a classification of all self-similar solutions to the CSF. Here self-similar means that the CSF evolves by a combination of rotation, translation and scaling. The first ancient solution which is not self-similar is the paperclip solution, also known as the Angenent oval (introduced in \cite{NLW}, \cite{A}). The Angenent oval is a convex embedded compact ancient solution lying inside a slab, which can be formed by gluing two Grim Reapers moving in opposite directions. More examples are constructed in \cite{AQ}, \cite{Y} by gluing finitely many Grim Reapers. 

Although ancient solutions to the CSF exist in abundance, there are only four convex embedded one up to ambient isometry and parabolic rescaling. Daskalopoulos, Hamilton, and Sesum show in \cite{DHS} that the shrinking circle and the Angenent oval are the only compact embedded ancient solutions. Recently, Bourni, Langford and Tinaglia show in \cite{BLT} that the stationary line and the Grim Reaper are the only noncompact convex embedded example. Thus the classification of ancient embedded convex solutions to the CSF is complete. 

On the other hand, from the constructions in \cite{AQ}, \cite{Y}, \cite{Ha}, it is unlikely to obtain a complete classification of all ancient solutions. One may wonder what geometric conditions can be imposed in order to obtain some classification results. 

Recently, Chini and M{\o}ller prove in \cite{CM} a bi-halfspace property for ancient codimension one MCF: Let $\{M_t : t\in (-\infty, T_0)\}$ be an $n$-dimensional properly immersed ancient solution to the MCF in $\mathbb R^{n+1}$. For each $t$, let $\operatorname{Conv}(M_t)$ be the convex hull of $M_t$ in $\mathbb R^{n+1}$. Then up to ambient isometry and parabolic rescaling, 
\begin{equation}
\mathscr M:= \bigcup_{t < T_0} \operatorname{Conv}(M_t) 
\end{equation}
is either 
\begin{itemize}
\item the hyperplane $\mathbb L= \{ x \in \mathbb R^{n+1} : x_{n+1} = 0\}$,
\item the slab $\mathbb S = \{x\in \mathbb R^{n+1} : 0<x_{n+1}<1\}$, 
\item the halfspace $\mathbb H = \{ x\in \mathbb R^{n+1}: x_{n+1} >0\}$
\item or the whole $\mathbb R^{n+1}$. 
\end{itemize}

As a result, the set of all properly immersed ancient solutions is divided into four classes, and one might ask if one can classify anyone of these. The stationary hyperplane is the only one with $\mathscr M =\mathbb L$. On the other hand, there are lots of ancient solutions with $\mathscr M = \mathbb R^{n+1}$, which (e.g.) include all properly immersed self-shrinking solutions \cite{CE}. 

For the remaining cases, note that the Grim Reaper and the Angenent oval have $\mathscr M = \mathbb S$, so are the examples in \cite{AQ}. In higher dimensions, compact convex ancient solutions of the MCF lying inside a slab is constructed in \cite{BLT0}, \cite{Wang}. Indeed, Wang shows in \cite{Wang} that for any convex ancient solutions, $\mathscr M$ is either a slab or the whole space. 

It is interesting that, to the author's knowledge (see also p.4 in \cite{CM}), there isn't any example of compact ancient solution which satisfies $\mathscr M = \mathbb H$. The main goal of this paper is to construct an infinite family of such examples in the case of CSF.

\begin{thm} \label{Compact example}
There is an infinite family of compact immersed ancient solutions $\{\mathcal C(t)\}_{t<T_0}$ to the CSF so that 
$$\bigcup_{t<T_0} \operatorname{Conv} (\mathcal C(t))$$
is the halfspace. Moreover, $\mathcal C(t)$ becomes embedded at some time before extinction.
\end{thm}

In the above theorem, we consider only compact solutions since noncompact examples can be constructed easily by gluing infinitely many Grim Reapers. This can be done using similar techniques in \cite{Y}, although that is not explicitly written down.

The examples constructed in Theorem \ref{Compact example} are not embedded. It is not known if there is any embedded compact ancient solution to CSF with $\mathscr M =\mathbb H$, or if the Angenent oval is the only compact embedded example inside a slab.

Next we briefly describe the proof of Theorem \ref{Compact example}. Recall that in \cite{AQ}, \cite{Y}, Angenent and You construct examples of ancient solutions to the CSF by gluing either finitely or infinitely many Grim Reapers along the common asymptotes. Thus their examples either lie in a slab, or is non-compact. In our situation we modify their construction. Roughly speaking, the ancient solutions constructed in Theorem \ref{Compact example} have the following property: there is a sequence of compact ancient solutions to the CSF $\{ \mathcal C_n(t)\}$ so that $\mathcal C_0(t)$ is the Angenent oval, and there is a sequence of time 
$$T_1 > \cdots > T_n > \cdots , \ \ T_n \to -\infty, $$
so that for each $n\in \mathbb N$ (see Figure 1), 
\begin{itemize}
\item when $t << T_n$, $\mathcal C_{n}(t)$ is formed by gluing a figure 8 (modeled by gluing four Grim Reapers) to $\mathcal C_{n-1}(t)$, 
\item when $t \sim T_n$, the figure 8 in $\mathcal C_n(t)$ disappears, and $\mathcal C_n(t)$ becomes embedded locally, and 
\item when $t>> T_n$, $\mathcal C_n(t)$ is close to $\mathcal C_{n-1}(t)$. 
\end{itemize} 
Then the ancient solutions constructed in Theorem \ref{Compact example} is the limit of $\{ \mathcal C_n(t)\}$ as $n\to \infty$. 

\bigskip

\begin{center}
\begin{tikzpicture}
\draw (-5,-0.3) .. controls (3,-0.3) .. (3,0);
\draw (5,-0.3) .. controls (-3,-0.3) .. (-3,0);
\draw (3,0) .. controls (3,0.3) and (-4, 0.3).. (-4,0.6);
\draw (-3,0) .. controls (-3,0.3) and (4, 0.3).. (4,0.6);\draw (4,0.6) .. controls (4, 0.9) .. (0,0.9);
\draw (-4,0.6) .. controls (-4, 0.9) .. (0,0.9);\end{tikzpicture}
\end{center}
\bigskip

\begin{center}
\begin{tikzpicture}
\draw (-5,-0.3) .. controls (1,-0.3) .. (1,0);
\draw (5,-0.3) .. controls (-1,-0.3) .. (-1,0);
\draw (1,0) .. controls (1,0.3) and (-2, 0.3).. (-2,0.6);
\draw (-1,0) .. controls (-1,0.3) and (2, 0.3).. (2,0.6);\draw (2,0.6) .. controls (2, 0.9) .. (0,0.9);
\draw (-2,0.6) .. controls (-2, 0.9) .. (0,0.9);\end{tikzpicture}
\end{center}

\bigskip

\begin{center}
\begin{tikzpicture}
\draw (-5,-0.3) .. controls (0,-0.3) .. (0,-0.1);
\draw (5,-0.3) .. controls (0,-0.3) .. (0,-0.1);
\draw (0,-0.1) .. controls (0,0.1) and (-1, 0.1).. (-1,0.3);
\draw (0,-0.1) .. controls (0,0.1) and (1, 0.1).. (1,0.3);
\draw (-1,0.3) .. controls (-1,0.7) .. (0,0.7);
\draw (1,0.3) .. controls (1,0.7) .. (0,0.7);
\end{tikzpicture}
\end{center}

\bigskip

\begin{center}
\begin{tikzpicture}
\draw (-5,-0.3) .. controls (-0.2,-0.3) .. (-0.2,-0.1);
\draw (5,-0.3) .. controls (0.2,-0.3) .. (0.2,-0.1);
\draw (-0.2,-0.1) .. controls (-0.2,0.1) and (-0.7, 0.1).. (-0.7,0.3);
\draw (0.2,-0.1) .. controls (0.2,0.1) and (0.7, 0.1).. (0.7,0.3);
\draw (-0.7,0.3) .. controls (-0.7,0.6) .. (0,0.6);
\draw (0.7,0.3) .. controls (0.7,0.6) .. (0,0.6);
\end{tikzpicture}
\end{center}

\bigskip

\begin{center}
\begin{tikzpicture}
\draw (-5,-0.3) .. controls (-0.6,-0.4) and (-0.6, 0.5) .. (0,0.5);
\draw (5,-0.3) .. controls (0.6,-0.4) and (0.6, 0.5)  .. (0,0.5);
\end{tikzpicture}

\bigskip

\begin{center}
\begin{tikzpicture}
\draw (-5,-0.3) .. controls (-3,-0.3) and (-1, 0.1) .. (0,0.1);
\draw (5,-0.3) .. controls (3,-0.3) and (1, 0.1)  .. (0,0.1);
\end{tikzpicture}
\end{center}

Figure 1: A figure 8 disappears around time $T_n$. 
\end{center}

\bigskip
Next we compare the technical difficulties in proving Theorem \ref{Compact example} with those in \cite{AQ}, \cite{Y}. In their works, the $C^0$ error at the tip regions are controlled by proving the convexity there, which holds when $-t$ is large enough. In our situation, since we require the figure 8 to unfold and disappear, we need $C^0$ estimates also when $-t$ is not large. Also, the error is harder to control since we need to glue infinitely many Grim Reapers. 

To obtain the required $C^0$ estimates, we construct the sequence of ancient solutions $\{\mathcal C_n(t)\}$ in such a way that for each $t$, the (signed) area bounded by $\mathcal C_n(t)$ are uniformly bounded. Then we prove (in Lemma \ref{C^0 control of C}) a $C^0$ estimates directly using the area bound. We remark that some forms of area estimates is also essential in the construction in \cite{AQ}, \cite{Y}.

In section \ref{C_n(t)}, we describe the construction of the sequence of ancient solutions $\{\mathcal C_n(t)\}_{n=0}^\infty$ and prove some properties about them. In section \ref{C^0 estimates}, we derive some $C^0$ estimates between $\mathcal C_n(t)$, $\mathcal C_m(t)$ for $m\neq n$. In section \ref{Proof of Thm}, we prove Theorem \ref{Compact example}. 

{\bf Acknowledgement}: The author would like to thank Man Chun Lee, Niels Martin M{\o}ller and Natasa Sesum for the discussion and interest in this work. The author would also thank the referee for pointing out a mistake in the first draft and for the extensive comments which improved greatly the exposition of this paper.

\section{Compact ancient solutions $\mathcal C_n(t)$} \label{C_n(t)}
First we recall some basic facts about Curve Shortening Flow (CSF). For any immersed curve $\gamma : I \to \mathbb R^2$, the curvature is given by 
$$\kappa = \langle \nabla_ e e, \nu\rangle, $$
where $e = \gamma'/|\gamma'|$ is the unit vector field along $\gamma$ and $\nu$ is the unit normal vector fields. The CSF is a family of immersed curve $\{ \gamma_t : I \to \mathbb R^2: t\in [t_1, t_2]\}$ so that  
\begin{equation*}
\frac{\partial \gamma_t}{\partial t} = \vec \kappa,
\end{equation*}
where $\vec \kappa = \kappa \nu$ is the curvature vector. The CSF is the one dimensional version of the Mean Curvature Flow (MCF).

When an immersed curve $\gamma$ is given by a graph $u$, the curvature is given by 
\begin{equation} \label{curvature represented as graph}
\kappa = \frac{u''}{(1+ u'^2 )^{3/2}},
\end{equation}
and when $\gamma_t$ is a family of immersed curves given by a graph $u = u(t, \cdot)$, then up to tangential diffeomorphisms, the CSF equation is equivalent to 
\begin{equation} \label{CSF equation by graph}
\frac{\partial u}{\partial t} = \frac{ u''}{1+ u'^2}.
\end{equation}

(\ref{CSF equation by graph}) has the following essential consequence: if $\gamma_t$ is a family of CSF which is given by a graph in $y$, that is, $\gamma_t(y) = (u(t, y), y)$ for some $ u(t, \cdot) : [a(t), b(t)] \to \mathbb R$ with $u(t, a(t)) = u(t, b(t))=0$, then 

\begin{equation} \label{derivative of area under graph}
\frac{d}{dt} \int _{a(t)}^{b(t)} u(t, y) dy = \theta (b(t)) - \theta (a(t)),
\end{equation}
where $\theta (y) = \arctan u' (y)$ is the angle the graph $u$ made with the $y$-axis. 

Next we recall the construction of a sequence $\{\mathcal C_n(t)\}_{n=1}^\infty$ of compact ancient solutions to the CSF in $\mathbb R^2$ formed by gluing finitely many Grim Reapers in \cite{AQ}. 

Let $\{a_n\}_{n=1}^\infty$ be any sequence of positive real numbers so that 
\begin{itemize}
 \item $a_n \ge 1$ for all $n\in \mathbb N$,
 \item $\sum a_n^{-1}$ diverges, and 
 \item $\sum a_n^{-2}$ converges. 
\end{itemize}
For example, $a_n = n$ satisfies the above conditions. Define $h_0 = 0$ and 
\begin{equation}
h_n = a_1^{-1} + a_2^{-1}+ \cdots + a_n^{-1},  \ \ \ \text{ for all } n\in \mathbb N. 
\end{equation}
Let $G: (0,\pi)\to \mathbb R$ be given by
\begin{align*}
G(y) = \ln\sin y.
\end{align*}
Then 
\begin{align*} 
\mathcal G(t):\ \{x =  G( y+\pi) -t\} 
\end{align*}
is a translating solution to the CSF. Any ambient isometry and parabolic rescaling of $\mathcal G_0(t)$ is called a Grim Reaper. 

For each $m\in \mathbb N$, define
\begin{align*}
G^-_m (y) &=  \frac{1}{a_m} G \left(a_m(y-2\pi h_{m-1})\right),\\
G^+_m(y) &= \frac{1}{a_m} G \left(a_m\left(y-2\pi h_{m-1} - a^{-1}_m \pi \right)\right).
\end{align*} 
and
\begin{align*}
\mathcal G^-_m (t)  &: \{x = - G^-_m(y) + a_m t \},\\
\mathcal G^+_m (t)  &: \{ x =  G^+_m(y) - a_m t\}, 
\end{align*} 
Note that for each $m\in \mathbb N$, $\mathcal G^-_m(t)$ and $\mathcal G^+_m(t)$ are Grim Reapers lying in the slab $\{ 2\pi h_{m-1} <y < 2\pi h_{m-1} + a^{-1}_m\pi\}$ and $\{2\pi h_{m-1} +a^{-1}_m \pi<y < 2\pi h_{m}\}$ respectively.  

Let $n \in \mathbb N$ and let $B_1, \cdots, B_n$ and $L$ be any positive numbers. We write 
\begin{equation} \label{definition of C_n}
C_n = B_1+ \cdots + B_n
\end{equation}
and for any $t < -C_n-1$, consider the following (portions of) Grim Reapers: 
\begin{align} 
\label{gluing specific G's 1} &\{\mathcal G_0(t): x\ge 0\},\\
\label{gluing specific G's 2} &\{\mathcal G^-_m(t + C_m): x\le 0\},\  \ m=1, 2, \cdots, n.\\
\label{gluing specific G's 3} &\{\mathcal G^+_m(t + C_m - La_m^{-2}), x\ge 0\}, \ \ m=1, 2, \cdots, n.
\end{align}
Note that $\mathcal G_0(t), \mathcal G^-_1(t)$ shares a common asymptotes $\{y=0\}$ and for $m=1, \cdots, n$, $\mathcal G^-_m(t + C_m)$ shares the common asymptotes $\{y = 2\pi h_{m-1}\}$, $\{y = 2\pi h_{m-1}+\pi a_m^{-1}\}$ with  $\mathcal G^+_{m-1}(t + C_{m-1} - La_{m-1}^{-2})$ and $\mathcal G^+_m (t+C_m - La_m^{-2})$ respectively. 

In the following we glue the above Grim Reapers along their common asymptotes to form a broken curve as in section 3.2 of \cite{AQ}. 

\begin{rem} \label{gluing}
In \cite{AQ}, the gluing is done using smooth functions. However, to simplify our argument (so that maximum principle is easily applicable) we choose the following specific gluing: 
\begin{enumerate}
\item [(I)] The Grim Reaper (\ref{gluing specific G's 1}) intersects $\{x=1\}$ at two points $(1, y_1)$, $(1, y_2)$ with $y_1<y_2$. Join $(0,-\pi)$ to $(1, y_1)$ and $ (1, y_2)$ to $(0,0)$ respectively by straight lines. 
\item [(II)] For each $m=1, 2, \cdots, n$, the Grim Reaper (\ref{gluing specific G's 2}) intersects $\{x = -a_m^{-1}\}$ at $(-a_m^{-1}, z^-_1)$, $(-a_m^{-1}, z^-_2)$ with $z^-_1<z^-_2$. Join  $(0,2\pi h_{m-1})$ to $(-a_m^{-1}, z^-_1)$ and $ (-a_m^{-1}, z^-_2)$ to $(0,2\pi h_{m-1}+ a^{-1}_m\pi )$ respectively by straight lines. 
\item [(III)] For each $m=1, 2, \cdots, n$, the Grim Reaper (\ref{gluing specific G's 3}) intersects $\{x = a_m^{-1}\}$ at $(a_m^{-1}, z^+_1)$, $(a_m^{-1}, z^+_2)$ with $z^+_1<z^+_2$. Join $(0,2\pi h_{m-1}+a^{-1}_m\pi)$ to $(a_m^{-1}, z^+_1)$ and $ (a_m^{-1}, z^+_2)$ to $(0,2\pi h_{m} )$ respectively by straight lines.
\end{enumerate}
\end{rem}

The resulting curve formed by gluing the Grim Reapers (\ref{gluing specific G's 1}), (\ref{gluing specific G's 2}) and (\ref{gluing specific G's 3}) as in Remark \ref{gluing} is graphical in $y$. Hence the curve is given by the graph of a function $\overline g_n(t,\cdot) : (-\pi, 2\pi h_n) \to \mathbb R$.

\begin{center}
\tikzset{every picture/.style={line width=0.75pt}} 

\begin{tikzpicture}[x=0.75pt,y=0.75pt,yscale=-0.7,xscale=0.7]

\draw [color={rgb, 255:red, 0; green, 0; blue, 0 }  ,draw opacity=1 ][line width=1.5]    (272.3,200.4) .. controls (317.3,202.4) and (613.3,206.4) .. (626.3,208.4) .. controls (639.3,210.4) and (657.3,222.4) .. (657.3,239.4) .. controls (657.3,256.4) and (638.3,269.4) .. (628.3,271.4) .. controls (618.3,273.4) and (340.3,272.4) .. (273.3,273.4) ;
\draw [color={rgb, 255:red, 0; green, 0; blue, 0 }  ,draw opacity=1 ][line width=1.5]    (222.3,183.4) .. controls (179.3,185.4) and (55.64,179.03) .. (44.3,177.4) .. controls (32.96,175.77) and (14.23,174.85) .. (14.3,161.4) .. controls (14.37,147.95) and (31.82,150.73) .. (42.3,148.4) .. controls (52.78,146.07) and (166.3,142.4) .. (222.3,142.4) ;
\draw  [dash pattern={on 4.5pt off 4.5pt}]  (239.3,278.4) -- (240.29,66.4) ;
\draw [shift={(240.3,63.4)}, rotate = 450.27] [fill={rgb, 255:red, 0; green, 0; blue, 0 }  ][line width=0.08]  [draw opacity=0] (8.93,-4.29) -- (0,0) -- (8.93,4.29) -- cycle    ;
\draw [color={rgb, 255:red, 0; green, 0; blue, 0 }  ,draw opacity=1 ][line width=1.5]    (255.3,95.4) .. controls (294.53,97.69) and (506.67,98.55) .. (517.98,100.34) .. controls (529.29,102.13) and (542.41,100.78) .. (542.15,114.23) .. controls (541.9,127.68) and (535.76,129.15) .. (525.25,131.33) .. controls (514.73,133.52) and (312.07,134.26) .. (255.3,136.4) ;
\draw [color={rgb, 255:red, 126; green, 211; blue, 33 }  ,draw opacity=1 ][line width=1.5]    (239.3,194.4) -- (272.3,200.4) ;
\draw [color={rgb, 255:red, 126; green, 211; blue, 33 }  ,draw opacity=1 ][line width=1.5]    (239.3,278.4) -- (273.3,273.4) ;
\draw [color={rgb, 255:red, 248; green, 231; blue, 28 }  ,draw opacity=1 ][line width=1.5]    (222.3,183.4) -- (239.3,194.4) ;
\draw [color={rgb, 255:red, 248; green, 231; blue, 28 }  ,draw opacity=1 ][line width=1.5]    (222.3,142.4) -- (239.3,139.4) ;
\draw [color={rgb, 255:red, 208; green, 2; blue, 27 }  ,draw opacity=1 ][line width=1.5]    (239.3,139.4) -- (255.3,136.4) ;
\draw [color={rgb, 255:red, 208; green, 2; blue, 27 }  ,draw opacity=1 ][line width=1.5]    (239.3,93.4) -- (255.3,95.4) ;

\draw (236,40.4) node [anchor=north west][inner sep=0.75pt]    {$y$};
\draw (195,267.4) node [anchor=north west][inner sep=0.75pt]    {$-\pi $};
\draw (185,81.4) node [anchor=north west][inner sep=0.75pt]    {$2\pi h_{1}$};

\end{tikzpicture}

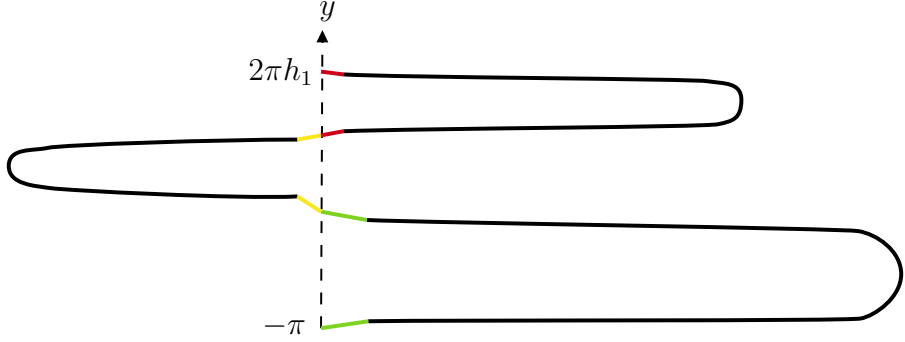
Figure 2: graph of $\bar g_n(t, \cdot)$ when $n=1$, with the gluing (I), (II), (III) marked in green, yellow and red respectively. 
\end{center}

For any $t < -C_n-1$, the broken solution $\overline{\mathcal C}_n(t)$ is defined as 
\begin{align*}
\overline{\mathcal C}_n(t) = \{x = \bar g_n(t, y) :-\pi \le y \le 2\pi h_n\}^{re},
\end{align*}
where for any subset $X\subset \mathbb R^2$ we write $X^{re} = X\cup RX$ and $R$ is the reflection along the $y$-axis (see Figure 3). To be precise, $\overline {\mathcal C}_n(t)$ consists of two graphs 
\begin{align*}
\{ x = \overline g_n(t,y) : -\pi \le y\le 2\pi h_n\},\ \ \{ x = -\overline g_n(t,y) : -\pi \le y\le 2\pi h_n\}
\end{align*} 
and we glue them at the common end points $(0,-\pi)$, $(0, 2\pi h_n)$. The resulting curve $\overline{\mathcal C}_n(t)$ is an image of a continuous mapping $\mathbb S^1 \to \mathbb R^2$ which is immersed away from finitely many points. 

\begin{center}\tikzset{every picture/.style={line width=0.75pt}} 

\begin{tikzpicture}[x=0.75pt,y=0.75pt,yscale=-0.7,xscale=0.7]

\draw [color={rgb, 255:red, 0; green, 0; blue, 0 }  ,draw opacity=1 ][line width=0.75]    (356.24,85.95) .. controls (386.76,87.46) and (587.46,90.5) .. (596.28,92.01) .. controls (605.09,93.53) and (617.3,102.62) .. (617.3,115.51) .. controls (617.3,128.4) and (604.42,138.25) .. (597.64,139.77) .. controls (590.86,141.28) and (402.35,140.52) .. (356.92,141.28) ;
\draw [color={rgb, 255:red, 0; green, 0; blue, 0 }  ,draw opacity=1 ][line width=0.75]    (322.34,73.06) .. controls (293.18,74.58) and (209.33,69.75) .. (201.64,68.51) .. controls (193.96,67.28) and (181.25,66.58) .. (181.3,56.39) .. controls (181.35,46.19) and (193.18,48.3) .. (200.29,46.53) .. controls (207.4,44.77) and (284.37,41.98) .. (322.34,41.98) ;
\draw [color={rgb, 255:red, 0; green, 0; blue, 0 }  ,draw opacity=1 ][line width=0.75]    (344.72,6.36) .. controls (371.32,8.1) and (515.16,8.75) .. (522.83,10.1) .. controls (530.5,11.46) and (539.4,10.44) .. (539.22,20.63) .. controls (539.05,30.82) and (534.89,31.94) .. (527.76,33.59) .. controls (520.63,35.25) and (383.21,35.81) .. (344.72,37.44) ;
\draw [color={rgb, 255:red, 0; green, 0; blue, 0 }  ,draw opacity=1 ][line width=0.75]    (333.87,81.4) -- (356.24,85.95) ;
\draw [color={rgb, 255:red, 0; green, 0; blue, 0 }  ,draw opacity=1 ][line width=0.75]    (333.87,145.07) -- (356.92,141.28) ;
\draw [color={rgb, 255:red, 0; green, 0; blue, 0 }  ,draw opacity=1 ][line width=0.75]    (322.34,73.06) -- (333.87,81.4) ;
\draw [color={rgb, 255:red, 0; green, 0; blue, 0 }  ,draw opacity=1 ][line width=0.75]    (322.34,41.98) -- (333.87,39.71) ;
\draw [color={rgb, 255:red, 0; green, 0; blue, 0 }  ,draw opacity=1 ][line width=0.75]    (333.87,39.71) -- (344.72,37.44) ;
\draw [color={rgb, 255:red, 0; green, 0; blue, 0 }  ,draw opacity=1 ][line width=0.75]    (333.87,4.84) -- (344.72,6.36) ;
\draw [color={rgb, 255:red, 0; green, 0; blue, 0 }  ,draw opacity=1 ][line width=0.75]    (346.21,42.59) .. controls (375.35,40.9) and (459.23,45.23) .. (466.93,46.42) .. controls (474.62,47.61) and (487.33,48.23) .. (487.34,58.43) .. controls (487.35,68.62) and (475.51,66.59) .. (468.41,68.39) .. controls (461.31,70.2) and (384.36,73.44) .. (346.39,73.66) ;
\draw [color={rgb, 255:red, 0; green, 0; blue, 0 }  ,draw opacity=1 ][line width=0.75]    (333.87,39.71) -- (346.21,42.59) ;
\draw [color={rgb, 255:red, 0; green, 0; blue, 0 }  ,draw opacity=1 ][line width=0.75]    (346.39,73.66) -- (333.87,81.4) ;
\draw [color={rgb, 255:red, 0; green, 0; blue, 0 }  ,draw opacity=1 ][line width=0.75]    (323.05,38.2) .. controls (296.44,36.49) and (152.6,35.96) .. (144.93,34.61) .. controls (137.26,33.27) and (128.36,34.29) .. (128.53,24.1) .. controls (128.69,13.91) and (132.85,12.79) .. (139.98,11.13) .. controls (147.11,9.47) and (284.52,8.78) .. (323.02,7.13) ;
\draw [color={rgb, 255:red, 0; green, 0; blue, 0 }  ,draw opacity=1 ][line width=0.75]    (333.87,4.84) -- (323.02,7.13) ;
\draw [color={rgb, 255:red, 0; green, 0; blue, 0 }  ,draw opacity=1 ][line width=0.75]    (333.87,39.71) -- (323.02,38.2) ;
\draw [color={rgb, 255:red, 0; green, 0; blue, 0 }  ,draw opacity=1 ][line width=0.75]    (311.32,140.46) .. controls (280.82,138.86) and (80.12,135.27) .. (71.31,133.73) .. controls (62.49,132.19) and (50.31,123.06) .. (50.35,110.17) .. controls (50.39,97.28) and (63.3,87.47) .. (70.08,85.97) .. controls (76.87,84.47) and (265.37,85.76) .. (310.8,85.13) ;
\draw [color={rgb, 255:red, 0; green, 0; blue, 0 }  ,draw opacity=1 ][line width=0.75]    (333.87,145.07) -- (311.5,140.46) ;
\draw [color={rgb, 255:red, 0; green, 0; blue, 0 }  ,draw opacity=1 ][line width=0.75]    (333.87,81.4) -- (310.8,85.13) ;

\end{tikzpicture}

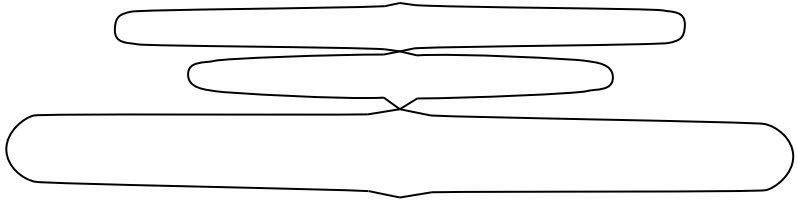
Figure 3: the broken curve $\overline{\mathcal C}_n (t)$ when $n=1$. 
\end{center}

For each $\alpha >C_n+1$, let $\mathcal C^\alpha_n (t)$ be the CSF with $\mathcal C_n^\alpha(-\alpha) = \overline{\mathcal C}_n (-\alpha)$. 

\begin{rem} \label{CSF starting at Lipschitz curve} The CSF starting at a piecewise smooth curve $\gamma : \mathbb S^1 \to \mathbb R^2$ is constructed in \cite{A91} as follows: let $\{\gamma_k: \mathbb S^1 \to \mathbb R^2\}$ be a sequence of smooth immersions with bounded derivatives which converges in $C^0$ to $\gamma$. For each $k$, we consider the CSF $\gamma_k(t)$ starting at $\gamma_k$. For small $t>0$, the interior estimates \cite[Theorem 3.1]{A91} enable us to take the limit of $\gamma_k(t)$ as $k\to \infty$ and construct a CSF $\{ \gamma(t)\}$ starting at $\gamma$. The non-compact situation is done similarly in \cite{EH}. 
\end{rem}

Note that $C^\alpha_n(t)$ are invariant under the reflection $R$. We write 
\begin{align}
\mathcal C_n^\alpha (t) = \left\{ x = g^\alpha_n (t,y) : y^{\alpha,-}_{0,n}(t)\le y\le y^{\alpha, +}_{n,n}(t)\right\} ^{re}.
\end{align}

Here $g^\alpha_n$ satisfies the graphical CSF equation (\ref{CSF equation by graph}) and $g^\alpha_n(-\alpha, \cdot) = \bar g_n(-\alpha, \cdot)$. Note that $\mathcal C^\alpha_n(t)$ is defined as least when $-\alpha < t<-C_n-1$ and intersects the $y$-axis at $2n+2$ points 
\begin{equation} \label{zeros of C^alpha_n}
 -\pi < y^{\alpha, -}_{0,n}(t) < y^{\alpha,+}_{0,n}(t) < \cdots < y^{\alpha, -}_{n,n}(t) < y^{\alpha, +}_{n,n}(t)< 2\pi h_n.
\end{equation}
$2n$ of them (besides $(0,y^{\alpha, -}_{0,n}(t))$, $(0, y^{\alpha,+}_{n,n}(t)))$ are also self-intersection points of $\mathcal C^\alpha_n(t)$.

The following theorem is proved in section 4 of \cite{AQ}. 

\begin{thm} \label{C_n(t) ancient}
For each $n\in \mathbb N$, the CSFs $\mathcal C^\alpha_n(t)$ converges, as $\alpha\to +\infty$, to an ancient solution $\mathcal C_n(t)$ to the CSF. 
\end{thm}

Just like $\mathcal C^\alpha_n(t)$, the immersion $\mathcal C_n(t)$ is defined when $t<-C_n-1$ and intersects the $y$-axis at $2n+2$ points 
$$ -\pi < y^-_{0,n}(t) < y^+_{0,n}(t) < \cdots < y^-_{n,n}(t) < y^+_{n,n}(t)< 2\pi h_n.$$
$2n$ of them (besides $(0,y^-_{0,n}(t))$, $(0, y^+_{n,n}(t)))$ are also self-intersection points of $\mathcal C_n(t)$. 

For each $n \in \mathbb N$ we write 
$$ \mathcal C_n(t) = \{ x = g_n(t,y) :  y^-_{0,n}(t) \le y\le y^+_{n,n}(t)\}^{re},$$
where
\begin{equation} \label{definition of g_n}
g_n = \lim_{\alpha \to \infty} g^\alpha_n.
\end{equation}

\begin{lem} \label{g_m ge g_n}
If $m>  n$, then $g_m(t,y) > g_n(t,y)$ for all $(t,y)$ so that both $g_n, g_m$ are defined.   
\end{lem}

\begin{proof}
For any $t < -C_m-1$, choose $\alpha  > -t$. Then 
\begin{align*}
C^\alpha_n(-\alpha) = \overline{\mathcal C}_n (-\alpha), \ \ C^\alpha_m(-\alpha) = \overline{\mathcal C}_m (-\alpha).
\end{align*}
Also, by Remark \ref{gluing}, $\overline g_m(-\alpha, y)=\overline g_n(-\alpha, y)$ whenever $\overline g_n(-\alpha, y)$ is defined. As in Remark \ref{CSF starting at Lipschitz curve}, we choose a sequence of function $f_{k} : [a_k, b_k] \to \mathbb R$ with the following properties:  
\begin{itemize}
\item $-\pi < a_k <b_k< 2\pi h_n$,  
\item $a_k \to -\pi$ and $b_k \to 2\pi h_n$ as $k\to \infty$,
\item $f_k(a_k) = f_k(b_k) = 0$,  
\item $f_k (y)< \overline g_n (-\alpha, y)$ for all $y\in [a_k, b_k]$,
\item for each $k$, 
\begin{equation}\label{approximated immersion of f_k}
 \{ x = f_k(y) :  y\in [a_k, b_k]\}^{re}
 \end{equation}
is a smooth immersion, and
\item $f_k(\cdot)\to \bar g_n(-\alpha,\cdot)$ in $C^0$ as $k\to \infty$.
\end{itemize}
For each $k$, let 
\begin{equation}
\{x= f_k(t, y) : a_k(t)\le y\le b_k(t)\}^{re}
\end{equation}
be the CSF starting at (\ref{approximated immersion of f_k}) with $f_k(-\alpha, \cdot) = f_k(\cdot)$. Then $f_k(t, y)$ satisfies the graphical CSF equation (\ref{CSF equation by graph}). Note that $\bar g_m (-\alpha, \cdot)>f_k(\cdot)$. For each $k$ and for all $t>-\alpha$, we claim that 
\begin{equation} \label{f_k (t, )< g^alpha (t, )}
g^\alpha_m (t, y) >f_k(t, y) 
\end{equation}
for all $y \in [a_k(t), b_k(t)]$. The inequality is obvious when $t$ is close to $-\alpha$. In general we argue by contradiction: assume that (\ref{f_k (t, )< g^alpha (t, )}) fails at the first instance $t_0$. Since (\ref{f_k (t, )< g^alpha (t, )}) holds for all $t<t_0$, the graph of $f_k(t_0, \cdot)$ touches the graph of $g^\alpha_m (t_0, \cdot)$ from below. By \cite[Theorem 1.1]{A91}, these two graphs can touch at only finitely many points. Since $f_k$, $g^\alpha_m$ both satisfy (\ref{CSF equation by graph}), the strong maximum principle implies that the graph of $f_k(t_0, \cdot)$ and $g^\alpha_m (t_0, \cdot)$ cannot touch in the interior $(a_k(t_0), b_k(t_0))$. Similarly, represent $f_k$, (resp. $g^\alpha_m$) locally around $(0,a_k(t_0))$ (resp. $(0,y^{\alpha,-}_{0,m}(t))$) as a graph of $x$, the strong maximum principle implies that the two graphs cannot touch at $(0,a_k(t_0))$, $(0,y^{\alpha, -}_{0,m}(t_0))$. Lastly, if the graphs touch at $(0,b(t_0))$, then either 
$$\lim_{y\to b_k(t_0)^-} (g_m^\alpha)'(t_0, y) = -\infty$$
(when $b_k(t_0) = y^{\alpha, +}_{m,m}(t_0)$) or 
$$(g^\alpha_m)' (t_0, b_k(t_0)) \le 0,$$
(when $b_k(t_0) = y^{\alpha, +}_{k,m}(t_0)$ for some $k=n, \cdots, m-1$). Representing $f_k$, $g^\alpha_m$ locally around $(0,b_k(t_0))$ as graphs of $x$, one sees that the first case is impossible by the strong maximum principle. For the second case, using also
$$\lim_{y\to b_k(t_0)^-} f'_k(t_0, y)=-\infty,$$
there is $y< b_k(t_0)$ so that $g(t_0, y) < f_k(t_0, y)$, which is again impossible. Thus the graph of $g^\alpha_m(t_0, \cdot)$ and $f_k(t_0, \cdot)$ do not touch and this contradicts to the choice of $t_0$. This finishes the proof of (\ref{f_k (t, )< g^alpha (t, )}). 

Taking $k\to \infty$, we have 
$$g^\alpha_m (t,y)\ge g^\alpha_n(t,y)$$ 
for all $\alpha > -t$. By (\ref{definition of g_n}), we have $g_m(t, y) \ge g_n(t, y)$. An argument using \cite[Theorem 1.1]{A91} and strong maximum principle again imply $g_m(t, y) > g_n(t, y)$ whenever $g_n(t, y), g_m(t, y)$ are defined. 
\end{proof}

\begin{lem} \label{g_m < G^-}
Let $n\in \mathbb N$ and $m \ge n$. If $t < -C_n-1$ and $g_m(t, y)$ is defined at $t$, then 
$$g_m(t, y)< -G^-_n (y) + a_n (t+C_n)$$
for all $2\pi h_{n-1} \le  y\le 2\pi h_{n-1} + \pi a_n^{-1}$.
\end{lem}

\begin{proof}
Let $-\alpha < t$. By definition, 
$$\bar g_m(-\alpha, y) \le -G^-_n (y) + a_n (t+C_n)$$ 
for all $y\in (2\pi h_{n-1}, 2\pi h_{n-1}+ a_n^{-1} \pi)$. Since $-G^-_n (y) + a_n (t+C_n)$ converges to $+\infty$ as $y\to 2\pi h_{n-1}$, $2\pi h_{n-1} + a_n^{-1} \pi$, $g^\alpha_m(t, y)$ is bounded and both $g^\alpha_m(t, y)$, $-G^-_n (y) + a_n (t+C_n)$ satisfy (\ref{CSF equation by graph}), the comparison principle \cite[Theorem 4.3]{MP} implies that 
$$ g^\alpha_m(t, y) \le -G^-_n(y) + a_n (t+C_n)$$
for all $-\alpha<t<-C_n-1$. Taking $\alpha \to \infty$, we have $g_m(t, y) \le G^+_n(y) - a_n (t+C_n)$ and the strictly inequality follows from strong maximum principle. 
\end{proof}

Next we define a sequence of ancient solutions $\mathcal O_n(t)$. For each $n\ge 0$, let
\begin{align} 
\overline{\mathcal O}_n(t) =\{x = \bar g_n(t, y) : 2\pi h_{n-1} + a_n^{-1} \pi \le y\le 2\pi h_n\}^{re}. 
\end{align}

As in the construction of $\mathcal C_n(t)$, we let $\alpha >0$ and let $\mathcal O^{\alpha}_n(t)$ be the CSF with $\mathcal O^{\alpha}_n (-\alpha) = \overline{\mathcal O}_n(-\alpha)$. Write 
\begin{equation}
\mathcal O^\alpha_n (t) = \{ x = o^\alpha_n (t, y): o^{\alpha,-}_n(t) \le y\le o^{\alpha,+}_n(t)\}^{re},  
\end{equation}
where $o^\alpha _n (t, \cdot)\ge 0$. Take $\alpha \to \infty$, we obtain the ancient solution $\mathcal O_n (t)$. Note that $\mathcal O_0(t)= \mathcal C_0(t)$ by construction. Write 
\begin{equation}
\mathcal O_n (t) = \{ x = o_n (t, y): o^-_n(t) \le y\le o^+_n(y)\}^{re},  
\end{equation}
where $o_n (t, y) \ge 0$. For $n\ge 1$, the strong maximum principle implies that 
\begin{equation}
o_n (t, y) < g_n(t, y)
\end{equation} 
whenever $t > -C_n-1$. 

Next we estimate the ``area'' bounded between 
\begin{itemize}
\item $o_n$ and $g_n$,
\item $g_m$ and $-G^-_n(\cdot) + a_n(t+C_n)$, and 
\item $g_n$ and $g_m$ 
\end{itemize}
respectively. To be precise, let $f_i : I_i \to \mathbb R$, $i=1, 2$ be two functions defined on two intervals with $I_1\subset I_2$, $f_1\le f_2$ and $f_i (\partial I_i) = 0$ for $i=1,2$. The {\sl signed area} bounded between $f_1$ and $f_2$ is defined as 
$$ \int_{I_2} f_2-f_1 ,$$
where we extend the domain of $f_1$ to $I_2$ by setting $f_1(x) = 0$ for all $x\in I_2\setminus I_1$. If $f_2\ge 0$ in $I_2\setminus I_1$, then the signed area is just the area under the graph $|f_2  -f_1|$, and in this case we simply use the word {\sl area}. 

\begin{lem} \label{area bound between g_n and o_n}
For each $n\ge 1$, there is $A_n >0$ depending on $\{a_n\}$ so that for all $t < -C_n-1$, the area bounded between $o_n(t, y)$ and $g_n(t, y)$ in $[ y^-_{n,n}(t), y^+_{n,n}(t)]$ is smaller than $A_n$.
\end{lem}

\begin{proof}
For any $\alpha > C_n +1$ and $t \in [-\alpha, -C_n-1]$, Let $A^\alpha (t)$ be the area bounded between $o^\alpha_n(t, \cdot)$ and $g^\alpha_n(t, \cdot)$ in $[y^{\alpha,-}_{n,n}(t), y^{\alpha, +}_{n,n}(t)]$. Note that 
\begin{equation} \label{A^alpha (-alpha) is o(1)}
A^\alpha(-\alpha) = 0,
\end{equation} 
Since $o^\alpha(-\alpha, y) = g^\alpha_n (-\alpha, y) = \bar g_n (-\alpha, y)$ for all $y\in [2\pi h_{n-1}+ a_n^{-1} \pi, 2\pi h_n]$. Using (\ref{derivative of area under graph}), we have
\begin{equation} \label{area growth formula}
\frac{d}{dt} A^\alpha (t) = \theta(t),
\end{equation}
here $\theta(t)$ is the angle $\{x = g^\alpha_n (t, y)\}$ made with the $x$-axis at $(0,y^{\alpha,-}_{n,n}(t))$. Arguing as in the proof of Lemma 3.5 in \cite{AQ}, we have $|\theta(t)| \le Me^{\delta (t+C_n)}$, where $M, \delta$ depends only on $a_n$. Together with (\ref{A^alpha (-alpha) is o(1)}) the lemma is proved by integrating (\ref{area growth formula}) and take $\alpha \to \infty$.
\end{proof}

The following Lemma is proved similarly as in Lemma 3.5 and 3.6 in \cite{AQ}. 

\begin{lem} \label{area between g_n and G^-_n}
There is $A'_n >0$ depending on $\{a_n\}$ so that for all $t < -C_n-1$, the area bounded between $\max\{ G^-_n (y) - a_n^{-1} (t+C_n),0\}$ and $|g_n(t, y)|$ in $[ y^+_{n-1,n}(t), y^-_{n,n}(t)]$ is smaller than $A'_n$.
\end{lem}

\begin{center}
\tikzset{every picture/.style={line width=0.75pt}} 

\begin{tikzpicture}[x=0.75pt,y=0.75pt,yscale=-.7,xscale=.7]

\draw [color={rgb, 255:red, 0; green, 0; blue, 0 }  ,draw opacity=1 ][fill={rgb, 255:red, 74; green, 144; blue, 226 }  ,fill opacity=0.4 ][line width=1.5]    (308.48,4) .. controls (343.99,6) and (577.58,10) .. (587.84,12) .. controls (598.1,14) and (612.3,26) .. (612.3,43) .. controls (612.3,60) and (597.31,73) .. (589.41,75) .. controls (581.52,77) and (362.17,83) .. (309.3,84) ;
\draw [fill={rgb, 255:red, 155; green, 155; blue, 155 }  ,fill opacity=0.29 ]   (308.48,4) -- (309.01,20) ;
\draw [fill={rgb, 255:red, 155; green, 155; blue, 155 }  ,fill opacity=0.29 ]   (309.67,66.98) -- (309.3,84) ;
\draw [color={rgb, 255:red, 208; green, 2; blue, 27 }  ,draw opacity=1 ][fill={rgb, 255:red, 255; green, 255; blue, 255 }  ,fill opacity=1 ][line width=1.5]    (309.01,20) .. controls (338.74,21.29) and (534.3,23.86) .. (542.89,25.15) .. controls (551.48,26.44) and (563.37,34.16) .. (563.37,45.1) .. controls (563.37,56.04) and (550.82,64.41) .. (544.21,65.69) .. controls (537.61,66.98) and (353.94,66.34) .. (309.67,66.98) ;
\draw [color={rgb, 255:red, 208; green, 2; blue, 27 }  ,draw opacity=1 ][fill={rgb, 255:red, 255; green, 255; blue, 255 }  ,fill opacity=1 ][line width=1.5]    (309.66,66.98) .. controls (279.93,65.68) and (84.36,63.02) .. (75.78,61.73) .. controls (67.19,60.44) and (55.3,52.71) .. (55.3,41.77) .. controls (55.3,30.83) and (67.86,22.47) .. (74.47,21.18) .. controls (81.07,19.9) and (264.74,20.62) .. (309.01,20) ;

\draw [fill={rgb, 255:red, 255; green, 255; blue, 255 }  ,fill opacity=1 ] [dash pattern={on 4.5pt off 4.5pt}]  (309.01,20) -- (309.67,66.98) ;
\draw [color={rgb, 255:red, 0; green, 0; blue, 0 }  ,draw opacity=1 ][fill={rgb, 255:red, 126; green, 211; blue, 33 }  ,fill opacity=0.45 ][line width=1.5]    (309.3,165.8) .. controls (279.83,162.99) and (83.77,153.96) .. (75.28,151.72) .. controls (66.8,149.49) and (55.84,140.72) .. (56.17,123.72) .. controls (56.49,106.73) and (66.2,98.54) .. (72.79,96.72) .. controls (79.39,94.9) and (265.35,83.79) .. (309.3,84) ;
\draw [color={rgb, 255:red, 189; green, 16; blue, 224 }  ,draw opacity=1 ][fill={rgb, 255:red, 255; green, 255; blue, 255 }  ,fill opacity=1 ][line width=1.5]    (627.3,93.8) .. controls (614,92.8) and (162.19,101.8) .. (122.3,104.8) .. controls (82.41,107.8) and (77.54,107.8) .. (77.54,123.8) .. controls (77.54,139.8) and (86.9,141.46) .. (121.3,144.8) .. controls (155.7,148.14) and (613.34,157.8) .. (628.3,157.8) ;
\draw    (309.3,84) -- (309.3,98.8) ;
\draw [fill={rgb, 255:red, 255; green, 255; blue, 255 }  ,fill opacity=1 ] [dash pattern={on 4.5pt off 4.5pt}]  (309.3,150.8) -- (309.3,98.8) ;
\draw    (309.3,150.8) -- (309.3,165.8) ;

\draw (5,6.4) node [anchor=north west][inner sep=0.75pt]  [color={rgb, 255:red, 208; green, 2; blue, 27 }  ,opacity=1 ]  {$\mathcal{O}_{n}( t)$};
\draw (565,125.4) node [anchor=north west][inner sep=0.75pt]  [color={rgb, 255:red, 189; green, 16; blue, 224 }  ,opacity=1 ]  {$\mathcal{G}^{-}_{n}( t+C_{n})$};

\end{tikzpicture}

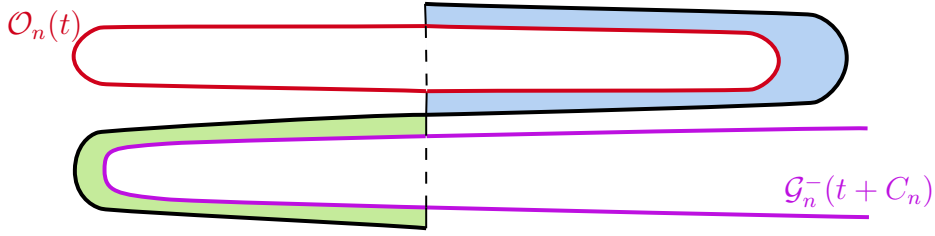
Figure 4: Area(Blue region)$\le A_n$, Area(Green region)$\le A_n'$. 
\end{center}
\begin{lem} \label{area between g_n and g_m}
For any $m, n\in \mathbb N$ with $m >n$, the signed area bounded between $g_n$ and $g_m$ is 
\begin{equation}
 L (a_{n+1}^{-2} + \cdots + a_m^{-2})
\end{equation}
whenever both $g_n, g_m$ are defined.
\end{lem}

\begin{proof}
It suffices to show the case when $m=n+1$. We first calculate the area bounded between $g^\alpha_n(t, \cdot) $ and $g^\alpha_{n+1}(t, \cdot)$ for $t>-\alpha$ and then take $\alpha \to \infty$. Since $g^\alpha_n$, $g^\alpha_{n+1}$ both satisfy the CSF equation, the area bounded between $g^\alpha _n(t, \cdot)$ and $g^\alpha_{n+1}(t, \cdot)$ is independent of $t$ by (\ref{derivative of area under graph}). Thus it suffices to calculate the area $A^\alpha$ bounded between $\bar g_n (-\alpha, \cdot)$ and $ \bar g_{n+1}(-\alpha, \cdot)$. From the construction described in Remark \ref{gluing}, since $\bar g_n(-\alpha, \cdot) = \bar g_{n+1}(-\alpha, \cdot)$ on $[-\pi, 2\pi h_{n-1}]$, we have 
$$A^\alpha = L a_{n+1}^{-2} + o(1),$$
as $\alpha \to \infty$. Hence the lemma is proved by taking $\alpha \to \infty$. 
\end{proof}

Before we move to the next section, we prove a lemma which provides a $C^0$ estimates using the area bound. In the following, we assume that either $r(t, y) = G(y) - t$ or $r(t, y) = o_0(t, y)$.

\begin{lem} \label{C^0 control of C}
Given any $A>0$, there are $E, \underline M>0$ depending only on $A$ such that the following holds: For any $M \ge \underline M$ and $T\in \mathbb R$ (when $r(t, \cdot) = o_0(t,\cdot)$, we also assume $T \le T^a$, so that $\max_y o_0(T^a, y) >1$). Let 
$$\{ \mathcal C(t) :t\in  [T -M, T]\}$$ 
be a CSF of immersed curves in $\mathbb R^2$. Assume that for any $t\in [T-M, T]$, $\mathcal C(t)$ is represented as a graph 
\begin{equation}
 \{ x = g_{\mathcal C}(t, y) :  y^-(t) \le y\le y^+(t)]\}
\end{equation}
inside the region (here $R_T = \max_y r(T, y)$)
$$\mathcal R:= [R_T-1, \infty)\times I$$ 
with $g_{\mathcal C} (t, y^\pm (t) ) = R_T-1$ and $I$ is an interval with length $\le 4\pi$. Also, 
\begin{enumerate}
\item for each $t$, $g_{\mathcal C}(t, y)$ has exactly one local maximum and no local minimum (except at the boundary points $y^\pm (t)$),
\item $g_{\mathcal C}(t, y) > r(t, y)$ for all $t\in [T-M, T]$, and 
\item the area bounded between the two graphs $g_{\mathcal C}(t, y)$, $r(t, y)$ in $\mathcal R$ is smaller than $A$. 
\end{enumerate}
Then $\max_y g_{\mathcal C}(T, y) - R_T<E$.  
\end{lem}

\begin{proof}
At time $t = T$, the graph $\{x = r(t, y) \}$ intersects with the line $\{ x = R_T -1 \}$ at two points $(R_T-1, c)$, $(R_T-1, d)$ with $c<d$. Note that $d-c\ge \ell$ for some positive constant $\ell$ (if $r(t, y) = G(y) -t$, $d-c$ is constant). Define
\begin{align}  \label{definition of K}
 K = \frac{2(A+5\pi)}{\ell}
 \end{align}
 and 
 \begin{align} \label{definition of E_n}
 E = \frac{20\pi K}{\ell}. 
\end{align}
Note that $K, E$ depend only on $A$.  
 
We argue by contradiction. Assume that $\max_y g_{\mathcal C} (T,y)-R_T >E$. From the CSF equation (\ref{CSF equation by graph}), the maximum value of $g_{\mathcal C}$ is non-increasing in $t$. As a result, 
$$\max_y g_{\mathcal C} (t,y)-R_T >E$$
for all $t\in [T-M, T]$ and in the region
$$ [R_T-1,R_T+E]  \times I,$$
the curve $\mathcal C(t)$ is a union of two graphs in $x$ with uniform $C^0$ bound by assumption (1). By the gradient estimates for graphical MCF \cite[Corollary 5.3]{ESIII}, assuming that $\underline M \ge 1$, the gradient of these two graphs are also bounded uniformly. Let $v_1(x), v_2(x)$ be those two graphs at time $t = T$ with $v_1<v_2$. 

Using assumption (2) and the Mean Value Theorem, there are $x_1, x_2\in ( R_T-1, R_T +E)$ so that 
$$ |v_i'(x_i)| \le 4\pi /E. $$
Then for all $x\in [R_T-1, R_T-1 +K]$, 
\begin{align*}
  |v_i'(x)| &\le |v_i'(x_i)| + \sup |v_i ''| |x-x_i| \\
  &\le 4\pi/ E + \sup|v_i''| E.
  \end{align*}
Next we use $|v'_i|\le C$, (\ref{curvature represented as graph}) for $i=1,2$ to obtain
\begin{align} \label{bounds on first derivatives} 
  |v_i'(x)| \le 4\pi /E + C\sup|\kappa_i| E. 
  \end{align}
Using the interior estimates \cite[Corollary 3.2 (i)]{EH}, we have
\begin{align} \label{interior estimates of kappa}
\sup |\kappa_i | \le \frac{C}{\sqrt {M}}
\end{align}
for some universal constant $C$. Now we choose 
\begin{equation} \label{definition of underline M_n}
 \underline M = CE^4, 
\end{equation}
here $C$ is a suitable universal constant so that when $M \ge \underline M$, (\ref{bounds on first derivatives}) and (\ref{interior estimates of kappa}) imply
\begin{align*}
 |v_i'(x)| \le 5\pi /E, \ \ \ \forall x\in [R_T -1, R_T-1+K]. 
\end{align*}
Hence we have 
$$ v_2-v_1 \ge d-c  - 10\pi K/E \ge  \frac{\ell}{2}$$
for all $x\in [R_T-1, R_T-1+K]$ by (\ref{definition of E_n}). By (\ref{definition of K}), we have 
$$\int _{R_T-1}^{R_T-1+K} (v_2-v_1) dx \ge \frac{\ell}{2}K=A+5\pi.$$
But this is impossible, since by assumption (3) and that $v_2-v_1 \le 4\pi$, 
\begin{align*}
A \ge \int _{R_T}^{R_T-1+K} (v_2-v_1) dx \ge \int _{-R_T-1}^{R_T-1+K} (v_2-v_1) - 4\pi.
\end{align*}
Thus we have arrived at a contradiction and finished the proof of the Lemma.
\end{proof}

\section{$C^0$-estimates between $\mathcal C_n(t)$, $\mathcal C_m(t)$} \label{C^0 estimates}
In this section we compare $\mathcal C_n(t)$ with $\mathcal C_m(t)$ for some particular time $t$. First we construct a sequence of ancient solutions to the CSF $\{\mathcal B_n(t)\}$, which are used to show the embeddedness of $\mathcal C_n(t)$. 

For any $L>0$ and $t <-1$, let $\overline b(t,\cdot): (0, 2\pi)\to \mathbb R$ be formed by gluing two Grim Reapers similar as in Remark \ref{gluing}. That is,
\begin{itemize}
\item $\overline b(t,y) = -G(y) + t$ if $\arcsin e^{1+t}\le y\le \pi - \arcsin e^{1+t}$, 
\item $\overline b(t,y) = G(y-\pi) - (t - L)$ if $\pi + \arcsin e^{1+t-L}\le y \le 2\pi - \arcsin e^{1+t-L}$, and
\item $\overline b(t,y)$ is linear in the intervals $(0, \arcsin e^{1+t})$, $(\pi - \arcsin e^{1+t}, \pi)$, $(\pi, \pi + \arcsin e^{1+t-L})$, and $(2\pi - \arcsin e^{1+t-L},2\pi)$, joining
\begin{itemize}
\item $(0,0)$ to $(-1, \arcsin e^{1+t})$, 
\item $(-1, \pi - \arcsin e^{1+t})$ to $(0,\pi)$,
\item $(0,\pi)$ to $(1, \pi + \arcsin e^{1+t-L})$, and 
\item $(1, 2\pi - \arcsin e^{1+t-L})$ to $(0,2\pi)$.  
\end{itemize}
\end{itemize}

Let $\overline{\mathcal B} (t)$ be the curve
\begin{equation}
\overline{\mathcal B}(t)= \big(\{ \overline b(t,y) : 0\le y\le 2\pi\} \cup \{ (x ,0): x\ge 0\}\big)^{re}.
\end{equation}

For any $\alpha >1$, let $\mathcal B^\alpha (t)$ be the CSF with $\mathcal B^\alpha (-\alpha) = \overline{\mathcal B}(-\alpha)$. For each $t > -\alpha$, $\mathcal B^\alpha(t)$ is a union of two graphs in $y$. Hence we write 
\begin{equation}
\mathcal B^\alpha (t) = \{ x = b^\alpha (t,y) : 0< y \le b^\alpha(t)\}^{re},
\end{equation}
where $b^\alpha(t, \cdot)$ satisfies the CSF (\ref{CSF equation by graph}) and $b^\alpha (-\alpha, y) = \bar b(-\alpha, y)$ for all $y>0$. 

As in \cite{AQ}, one can show that $\mathcal B^\alpha(t)$ converges to an ancient solution to the CSF $\mathcal B(t)$ as $\alpha \to\infty$. We write 
\begin{equation}
\mathcal B(t) = \{ x = b(t,y) : 0< y \le b(t)\}^{re},
\end{equation}
where 
$$ b(t, y) = \lim_{\alpha\to \infty} b^\alpha(t, y).$$

\bigskip

\begin{center}
\begin{tikzpicture}
\draw (-5,-0.3) .. controls (3,-0.3) .. (3,0);
\draw (5,-0.3) .. controls (-3,-0.3) .. (-3,0);
\draw (3,0) .. controls (3,0.3) and (-4, 0.3).. (-4,0.6);
\draw (-3,0) .. controls (-3,0.3) and (4, 0.3).. (4,0.6);\draw (4,0.6) .. controls (4, 0.9) .. (0,0.9);
\draw (-4,0.6) .. controls (-4, 0.9) .. (0,0.9);\end{tikzpicture}

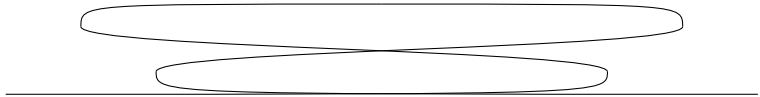
Figure 5: the barrier $\mathcal B(t)$.
\end{center}
\bigskip

Arguing as in Lemma \ref{area between g_n and g_m}, one has 
\begin{equation}
 \int_0^{b(t)} b(t, y) dy = L
\end{equation}
for all $t$ so that $b(t, \cdot)$ is defined. 

\begin{lem} \label{B(t) becomes embedded}
There is $\overline L>0$ so that whenever $L> \overline L$, $\mathcal B(t)$ becomes embedded at some time (and thus is defined for all $t$). 
\end{lem}

\begin{proof}
Observe that $\mathcal B(t)$ becomes embedded if $b(t, y) >0$ for all $y\in (0, b(t))$. Also, if $b_1(t, y), b_2(t, y)$ are constructed from above using $L=L_1, L_2$ respectively with $L_1<L_2$, then 
\begin{equation} \label{b_1 <b_2}
b_1(t, y) \le b_2(t, y)
\end{equation} 
by maximum principle. Thus it suffices to show the embeddedness for $L$ large enough. To this end, let $b_{min}$ be the global minimum of $b(-1, y)$,  where $b(t, y)$ is defined using a fixed $L_0$. We assume that $b_{min}<0$, or there is nothing to prove. Let $\mathcal G(t)$ be a Grim Reaper so that 
\begin{itemize}
 \item $\mathcal G(t)$ is symmetric about $\{y = \pi/2\}$ and is moving in the positive $x$ direction, and
 \item has width larger than $\pi$,
 \end{itemize}
 and if we represent $\mathcal G(t)$ as a graph of the function $\{x=g_{t} (y)\}$, 
 \begin{itemize}
 \item the minimum of $g_{-1}(y)$ is smaller than $b_{min}$, and 
 \item the graphs of $g_{-1}(y)$ and $b(-1, y)$ intersect at only one point, which is in the region $\{y >3\pi/2\}$.
\end{itemize}

Using (\ref{b_1 <b_2}), one sees that the same conditions are satisfied by $\mathcal G(t)$ if $b(t, y)$ is defined using any $L \ge L_0$. Being a translator, $\mathcal G(t)$ leaves the region $\{ x\le 0\}$ at some time $T>-1$. 
Now we choose $L$ large enough (depending on $T$)  so that the following is true: there is an Angenent oval $\gamma$ in $\mathbb R^2$ so that 
$$\gamma \subset \{ (x, y)\in \mathbb R^2: x\in (\pi, 3\pi/2),  \max\{ g_{-1}(y), 1\} < y < b(-1, y) \},$$
$\gamma$ is symmetric about $\{y=y_0\}$ for some $y_0 \in (\pi, 3\pi/2)$ and the area enclosed by $\gamma$ is larger then $(T+1)/2\pi$. 

\begin{center}
\tikzset{every picture/.style={line width=0.75pt}} 

\begin{tikzpicture}[x=0.75pt,y=0.75pt,yscale=-0.5,xscale=0.5]

\draw [color={rgb, 255:red, 144; green, 19; blue, 254 }  ,draw opacity=1 ][line width=1.5]    (658.24,274.41) .. controls (587.18,277.88) and (66.06,272.67) .. (42.38,264) .. controls (18.69,255.32) and (6.33,234.5) .. (7.36,189.38) .. controls (8.39,144.26) and (20.75,130.38) .. (40.32,118.23) .. controls (59.88,106.09) and (528.48,93.94) .. (660.3,93.94) ;
\draw [color={rgb, 255:red, 126; green, 211; blue, 33 }  ,draw opacity=0.89 ][line width=1.5]    (659.27,239.7) .. controls (228.06,241.78) and (30.02,260.53) .. (35.17,210.2) .. controls (15.6,125.18) and (578.94,211.94) .. (584.09,113.03) .. controls (589.24,57.5) and (194.8,87) .. (157.72,81.79) ;
\draw [color={rgb, 255:red, 208; green, 24; blue, 2 }  ,draw opacity=1 ][line width=2.25]    (280.3,123.2) .. controls (289.3,115.2) and (377.3,111.2) .. (436.3,111.2) .. controls (495.3,111.2) and (481.3,111.2) .. (504.3,112.2) .. controls (527.3,113.2) and (576.3,116.2) .. (555.3,131.2) .. controls (534.3,146.2) and (326.3,143.2) .. (301.3,140.2) .. controls (276.3,137.2) and (269.04,129.08) .. (293.3,118.2) ;
\draw    (159.78,32.74) -- (159.78,328.2) ;
\draw [shift={(159.78,29.74)}, rotate = 90] [fill={rgb, 255:red, 0; green, 0; blue, 0 }  ][line width=0.08]  [draw opacity=0] (8.93,-4.29) -- (0,0) -- (8.93,4.29) -- cycle    ;
\draw  [dash pattern={on 4.5pt off 4.5pt}]  (160.3,132.2) -- (664.3,126.2) ;

\draw (147.49,4.44) node [anchor=north west][inner sep=0.75pt]    {$y$};
\draw (125,118.4) node [anchor=north west][inner sep=0.75pt]    {$y_{0}$};
\end{tikzpicture}

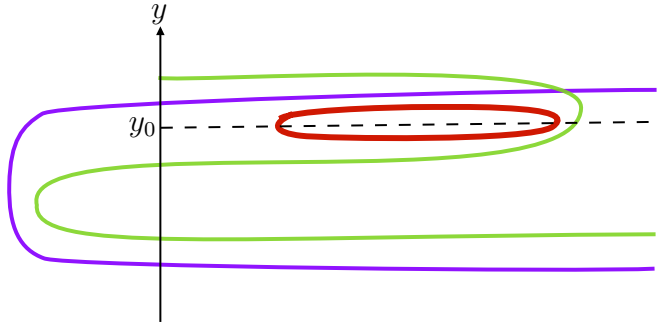
Figure 6: The Angenent oval $\gamma$ (in red), $\mathcal G(-1)$ (in purple) and $b(-1, y)$ (in green). 
\end{center}

Let $\gamma(t)$ be the CSF with $\gamma(-1) = \gamma$. Since $\{ \gamma (t)\}$ shrinks to a point and the area enclosed by $\gamma(t)$ decreases with a constant rate $2\pi$ \cite[Appendix B, Proposition 1]{B}, $\gamma(t)$ is defined at least up to time $T$. On the other hand, since $\gamma$ is disjoint from $\mathcal G(-1)$ and $\mathcal B(-1)$,  $\gamma(t)$ is also disjoint to $\mathcal G(t)$, $\mathcal B(t)$ for all $t\in [-1, T]$ \cite[Theorem 5.8]{MP}. Then $g_t(y_0) < b(t, y_0)$ for all $t\in [-1, T]$ and in particular $b(T, y_0)>0$ by the choice of $T$. This implies that $b(T, y) > 0$ for all $y \in [y_0, b(T))$. 

It remains to show that $b(T, y) >0$ for all $y\in (0,y_0)$. Since $b(t, y)\to +\infty$ as $y\to 0^+$, there is $y^0 >0$ so that $b(t, y) > g_t(y)$ for all $t\in [-1, T]$ and $y\in (0,y^0]$. In particular, we have $b (t, y^0)> g_t(y^0)$ for all $t\in [-1, T]$. Thus $b\ge g$ on the parabolic boundary of $ [-1, T] \times [y^0, y_0]$. By the comparison principle \cite[Theorem 4.3]{MP}, $b\ge g$ on $ [-1, T] \times [y^0, y_0]$. This finishes the proof since $g(T, y)$ is positive for all $y$. 

\end{proof}

From now on, we fix $L>\overline L $. Let $T^{e}$ be the time where $\mathcal B(t)$ becomes embedded. 

Next we define ${\mathcal B}_n(t)$ by parabolic rescaling and space time translations: Given any sequence $\{B_n\}_{n=1}^\infty$ of positive numbers, define 
\begin{equation} \label{barrier B_n(t) definition}
\mathcal B_n (t) = \frac{1}{a_n} \mathcal B(a_n^2 (t+C_n)) + (0,2\pi h_{n-1}), 
\end{equation}
Write 
\begin{equation}
\mathcal B_n(t) = \{ x = b_n(t,y) : 2\pi h_{n-1} < y\le b_n(t)\}^{re}. 
\end{equation}
Then
\begin{equation} \label{area bounded by B_n}
\int_{2\pi h_{n-1} }^{b_n(t)} b_n(t, y) dy =  L a_n^{-2}
\end{equation}
and $\mathcal B_n(t)$ becomes embedded when $a_n^2 (t+C_n) = T^{e}$, or 
\begin{equation} \label{B_n embedded time}
t= T^{e}_n := -C_n + T^{e}/a_n^2. 
\end{equation}

\begin{lem} \label{Lemma b_n >g_n}
For any $n\in \mathbb N$, $b_n (t, \cdot)> g_n( t, \cdot)$.
\end{lem}

\begin{proof}
Let 
\begin{align*}
\overline{\mathcal B}_n(t) = \frac{1}{a_n} \overline{\mathcal B} ( a_n^2 (t + C_n)) + (0,2\pi h_{n-1}).
\end{align*}
From the construction of $\mathcal B(t)$, one sees that $\mathcal B_n(t)$ is the limit of $\mathcal B^\alpha_n(t)$ as $\alpha \to +\infty$, where for any $\alpha > C_n + 1/a_n^2$, $\mathcal B^\alpha_n(t)$ is the CSF with 
\begin{align*}
\mathcal B^\alpha_n(-\alpha) = \overline{\mathcal B}_n(-\alpha).
\end{align*}
Also, from the construction of $\overline{\mathcal B}(t)$ and $\overline{\mathcal C}_n(t)$ in Remark \ref{gluing}, $\overline{\mathcal B}_n(t)=\overline{\mathcal C}_n(t)$ in $\{ y > 2\pi h_{n-1}\}$. 

Arguing similarly as in the proof of Lemma \ref{g_m ge g_n} and Remark \ref{CSF starting at Lipschitz curve}, $b^\alpha_n \ge g^\alpha_n$ in the region $y > 2\pi h_{n-1}$. Hence $b_n \ge g_n$ by taking $\alpha \to \infty$. This implies $b_n > g_n$ by the strong maximum principle.
\end{proof}

Note that we cannot argue directly by maximum principle and Lemma \ref{Lemma b_n >g_n} that $\mathcal C_n(t)$ becomes embedded in the region $\{ y > 2\pi h_{n-1}\}$. To show the embeddedness, we need to control the area bounded between $b_n$ and $g_n$ in that region.

\begin{prop} \label{embedded of C_n(t)}
Let $n$ be fixed. For all $\delta>0$, there is $B^e_n >0$ depending only on $a_{n-1}, a_n, \delta$ so that for all $B_n \ge B_n^e$, the immersed curve $\mathcal C_n(T^e_n + \delta)$ is embedded in the region $\{ y \ge 2\pi h_{n-1}\}$. 
\end{prop}

\begin{proof}
Let $M>0$. For any time $t <T^{e}_n$ and $\alpha >-t$, the graph $\{x = b^\alpha_n(t, y)\}$ intersects $\{x = M\}$ at either one or three points (depending on $t$). Let $y=b^\alpha_{n,M} (t)$ be the smallest $y$-coordinates of those points. Similarly, 
$$\{x = g^\alpha_n(t, y) :  y^{\alpha,-}_{n-1,n}(t)\le y\le y^{\alpha, +}_{n-1,n}(t)\}$$ 
intersect $\{x = M\}$ at two points and we let $y^{\alpha}_{n,M}(t)$ be the larger $y$-coordinates among those two. Then for any $\epsilon>0$, as in the proof of \cite[Lemma 3.5]{AQ}, one can find $M, B^e_n$ large so that if $B_n \ge B^e_n$, the region (see Figure 5) bounded by
\begin{itemize}
\item $\{ (M, y) :y^{\alpha}_{n,M}(t)\le  y\le b^\alpha_{n,M} (t) \}$, 
\item $\{ x = b_n^\alpha(t, y):  b^\alpha_{n,M} (t) \le y\le b^\alpha_n(t)\}$, 
\item $\{ (0,y): y^{\alpha, +}_{n,n} (t)\le y\le b^\alpha_n (t)\}$, and 
\item $\{x = g^\alpha_n(t, y): y^{\alpha}_{n,M}(t) \le y\le y^{\alpha, +}_{n,n} (t)\}$. 
\end{itemize} 
has area less than $\epsilon$. Taking $\alpha \to +\infty$, one sees that the area of the region bounded by
\begin{itemize}
\item $\{ (M, y) :y_{n,M}(t)\le  y\le b_{n,M} (t) \}$, 
\item $\{ x = b_n(t, y):  b_{n,M} (t) \le y\le b_n(t)\}$, 
\item $\{ (0,y): y^{+}_{n,n} (t)\le y\le b_n (t)\}$, and 
\item $\{x = g_n(t, y): y_{n,M}(t) \le y\le y^{+}_{n,n} (t)\}$. 
\end{itemize} 
is less than $\epsilon$ at all time $t$, where 
$$y_{n,M}(t):= \lim_{\alpha\to \infty} y^{\alpha}_{n, M}(t),\ \ b_{n,M} (t):=\lim_{\alpha \to \infty} b^\alpha_{n, M} (t).$$

\begin{center}

\tikzset{every picture/.style={line width=0.75pt}} 

\begin{tikzpicture}[x=0.75pt,y=0.75pt,yscale=-0.7,xscale=0.7]

\draw [fill={rgb, 255:red, 126; green, 211; blue, 33 }  ,fill opacity=0.25 ]   (150.3,89.8) .. controls (120.3,92.8) and (91.3,99.8) .. (74.3,109.8) .. controls (57.3,119.8) and (57.3,129.8) .. (56.3,134.8) .. controls (55.3,139.8) and (57.3,152.8) .. (70.3,160.8) .. controls (83.3,168.8) and (125.3,175.8) .. (149.3,176.8) ;
\draw [color={rgb, 255:red, 126; green, 211; blue, 33 }  ,draw opacity=1 ][fill={rgb, 255:red, 126; green, 211; blue, 33 }  ,fill opacity=0.25 ][line width=1.5]    (150.3,30.63) .. controls (177.3,30.8) and (405.3,33.8) .. (425.3,37.8) .. controls (445.3,41.8) and (456.3,46.8) .. (456.3,62.8) .. controls (456.3,78.8) and (458.3,81.8) .. (436.3,88.8) .. controls (414.3,95.8) and (234.3,99.8) .. (150.3,104.8) ;
\draw [color={rgb, 255:red, 126; green, 211; blue, 33 }  ,draw opacity=1 ][fill={rgb, 255:red, 255; green, 255; blue, 255 }  ,fill opacity=1 ][line width=1.5]    (150.3,104.8) .. controls (122.54,106.58) and (94.3,112.8) .. (87.3,116.8) .. controls (80.3,120.8) and (72.3,124.8) .. (73.3,136.8) .. controls (74.3,148.8) and (77.3,152.8) .. (97.52,157.63) .. controls (117.75,162.47) and (123.64,164.54) .. (149.64,165.54) ;
\draw [draw opacity=0][fill={rgb, 255:red, 126; green, 211; blue, 33 }  ,fill opacity=0.25 ]   (149.3,176.8) .. controls (201.3,179.8) and (626.3,198.8) .. (632.3,192.8) .. controls (638.3,186.8) and (632.3,187.8) .. (632.3,182.8) .. controls (632.3,177.8) and (620.3,180.8) .. (597.3,181.8) .. controls (574.3,182.8) and (171.3,168.8) .. (149.64,165.54) ;
\draw [fill={rgb, 255:red, 255; green, 255; blue, 255 }  ,fill opacity=1 ]   (150.3,41.86) .. controls (191.3,41.86) and (348.3,44.35) .. (392.3,46.85) .. controls (436.3,49.34) and (442.3,58.08) .. (442.3,66.81) .. controls (442.3,75.54) and (436.51,78.15) .. (410.3,80.53) .. controls (384.09,82.92) and (171.3,88.8) .. (150.3,89.8) ;
\draw [line width=1.5]    (150.3,41.86) -- (150.3,30.63) ;
\draw [color={rgb, 255:red, 128; green, 128; blue, 128 }  ,draw opacity=1 ] [dash pattern={on 4.5pt off 4.5pt}]  (632.3,203.8) -- (633.3,25.94) ;
\draw [color={rgb, 255:red, 128; green, 128; blue, 128 }  ,draw opacity=1 ] [dash pattern={on 4.5pt off 4.5pt}]  (151.3,194.8) -- (150.14,134.3) -- (150.46,75.3) -- (150.76,18.92) ;
\draw [color={rgb, 255:red, 126; green, 211; blue, 33 }  ,draw opacity=1 ][line width=1.5]    (149.64,165.54) .. controls (208.3,171.8) and (575.3,183.8) .. (632.3,182.8) ;
\draw    (149.3,176.8) .. controls (209.3,180.8) and (564.3,196.8) .. (632.3,192.8) ;
\draw    (632.3,182.8) -- (632.3,192.8) ;

\draw (129,3.11) node [anchor=north west][inner sep=0.75pt]    {$x=0$};
\draw (601,6.87) node [anchor=north west][inner sep=0.75pt]    {$x=M$};
\draw (450,20.91) node [anchor=north west][inner sep=0.75pt]  [color={rgb, 255:red, 126; green, 211; blue, 33 }  ,opacity=1 ]  {$b^{\alpha }_{n}( t,\ y)$};
\draw (5,78.06) node [anchor=north west][inner sep=0.75pt]    {$g^{\alpha }_{n}( t,\ y)$};

\end{tikzpicture}

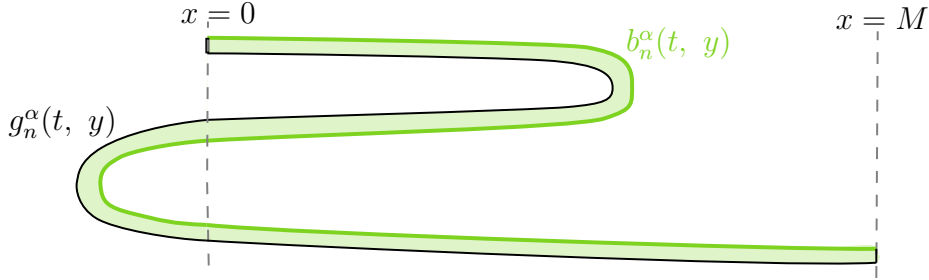
Figure 7: the area is less than $\epsilon$. 
\end{center}

By Lemma \ref{B(t) becomes embedded} and (\ref{barrier B_n(t) definition}), $\mathcal B_n(t)$ becomes embedded at some time. Then it is proved in \cite{Po} (see also \cite[Theroem 2.5]{H}) that $\mathcal B_n(t)$ are aymsptotic and converge to the straight line $\{ y=0\}$ as $t\to +\infty$. Thus the curvature of $\mathcal B_n(t)$ are uniformly bounded (say, by $C=C_{\mathcal B_n}$). Let $\epsilon < \pi/C^2$, Then for each $t\le T^{e}_n$, the ball of radius $\sqrt{\epsilon /\pi}$ touching the local maximum of $b_n( t,y)$ from below has area $\epsilon$. Thus $g_n(t,y)$ intersects this ball. In particular, for any $\delta >0$, one can choose $\epsilon$ small enough so that whenever $B_n \ge B^e_n$, the curve $\mathcal C_n(t)$ is defined at least up to $T^{e}_n+\delta$. 

At time $t=T^{e}_n$, the function $b_n(t,y)$ intersects $\{x=0\}$ at two points $b_n(T^e_n)$ and $b^e$ with $b^e < b_n(T^e_n)$. Let $\delta_1$ be a positive number so that
$$\delta_1 < \min \{ b^e - 2\pi h_{n-1}, b_n(T^e_n)-b^e\} .$$
By choosing a smaller $\epsilon$ (thus a larger $B^e_n$) if necessary, we assume 
$$\min \left\{ \int_{b^e-\delta_1} ^{b^e} b_n(T^e_n,y) dy, \int_{b^e}^{b^e+\delta_1} b_n(T^e_n,y) dy \right\}  >\epsilon.$$
The area estimates between $g_n$ and $b_n$ implies that $g_n(T^e_n,y)$ satisfies 
$$ b^e -\delta_1 < y^+_{n-1, n}(T^e_n) < b^e < y^-_{n,n} (T^e_n) < b^e+\delta_1,$$
that is, $g_n(T^e_n, \cdot)$ is positive outside $(b^e -\delta_1, b^e+\delta_1)$ in the region $\{ y\ge 2\pi h_{n-1}\}$. Thus it suffices to show that $g_n$ is positive inside $(b^e-\delta_1, b^e+\delta_1)$ at time $t = T^{e}_n + \delta$.
 
To see this, let $b_\delta$ be the local minimum of $b_n(T^e_n + \delta,y)$ in $(b^e-\delta_1, b^e+ \delta_1)$ (we assume that $\delta>0$ is small so that $b_\delta$ exists). Note that in $(b^e-\delta_1, b^e+\delta_1)$, by the gradient estimates \cite[Corollary 5.3]{ESIII}, we have 
\begin{equation} \label{gradient estimates locally}
|g'_n (T^e_n + \delta, \cdot)|\le C,
\end{equation}
for some universal constant $C$. Now we argue by contradiction that $g_n(T^n_e+\delta, \cdot)$ is positive in $(b^e-\delta_1, b^e+\delta_1)$: if not, then $g_n(T^e_n+\delta, y_0)\le 0$ for some $y_0\in (b^e-\delta_1, b^e+\delta_1)$. then $g_n(T^e_n+\delta,y) \le b_\delta/2$ for all $y\in (b^e-\delta_1, b^e+\delta_1)$ with $|y-y_0| < b_\delta/2C$ by (\ref{gradient estimates locally}). Thus 
$$ \int_{I} (b_n(t,y) - g_n(t,y) ) dy \ge b^2_{\delta}/8C.$$
for some interval $I\subset (b^e-\delta_1, b^e+\delta_1)$. But this is impossible if $\epsilon < b^2_\delta/8C$. Thus $g_n$ is positive at time $t = T^{e}_n+\delta$ for small enough $\epsilon$ in the region $\{ y\ge 2\pi h_{n-1}\}$. 

\end{proof}

From now on we assume that in the constructions of $\{\mathcal C_n(t)\}$ we have chosen $B_n \ge B^e_n$ for all $n\in \mathbb N$, where $B^e_n$ is chosen as in Proposition \ref{embedded of C_n(t)} with $\delta=1$.

\begin{cor} \label{C_n becomes embedded and defined up to T_0}
For each $n\in \mathbb N$, assume that $B_m \ge B_m^e$ for all $m=1, 2, \cdots, n$. Then $\mathcal C_n(t)$ becomes embedded and shrinks to a point. $\mathcal C_n(t)$ are all defined up to time $T_0$.
\end{cor}

\begin{proof}
From the choice of $B^e_n$, by Proposition \ref{embedded of C_n(t)} we have $g_n(t,y)>0$ in the region $y > 2\pi h_{n-1}$ after time $T^e_n+1$. Since $g_n (t,y) > g_k(t,y)$ when $n> k$, an inductive argument shows that $g_n$ becomes positive after some time and thus $\mathcal C_n(t)$ becomes embedded. Since $g_n > g_0$, this also shows that $\mathcal C_n(t)$ is defined up to time $T_0$. 
\end{proof}

\begin{dfn} \label{closedness definition}
For any $ t<-C_n-1$, we say that an immersed curve $\mathcal C$ is close to $\mathcal C_n(t)$ if the followings hold: 
\begin{enumerate}
\item $\mathcal C$ can be written as
$$ \mathcal C = \{ x= g_{\mathcal C} (y) :  y^-_0 \le y\le y^+_n\}^{re}$$
for some function $g_{\mathcal C} : [y^-_0, y^+_n] \to \mathbb R$ with $g_{\mathcal C}(y^-_0) = g_{\mathcal C} (y^+_n) = 0$. 
\item $2\pi h_n + \pi a^{-1}_{n+1}> y^+_n > y^-_0 > -\pi$;
\item $g_{\mathcal C} (y)\ge g_n(t,y)$ and the area bounded between $g_{\mathcal C}(y)$ and $g_n(t,y)$ is less than $L \sum_{k=n}^\infty a_k^{-2}$;
\item $g_{\mathcal C} (y) < G_m^- (y) + a_m(t+C_m)$ for $m=1, 2, \cdots, n$.
\item $g_{\mathcal C}$ has exactly $n+1$ local maximum and $n$ local minimum. 
\end{enumerate}
\end{dfn}

\begin{prop} \label{epsilon closedness - inductive prop}
Let $n \in \mathbb N$ be fixed. Then there are $\overline M_n, \underline M_n>0$ depending only on $\{a_n\}$ such that the following holds: For all $M_n\ge \underline M_n$ and $B_n \ge \max\{ \overline M_n, B^e_n\}$, let $\mathcal C$ be an immersed curve that is close to $\mathcal C_n(-C_n - M_n)$. Then $\mathcal C(-C_n-1+t)$ is close to $\mathcal C_{n-1}(-C_n-1 + t)$ whenever $t \in (\overline M_n, B_n)$. Here $\mathcal C(t)$ is the CSF with $\mathcal C(-C_n-M_n)= \mathcal C$.
\end{prop}

\begin{proof}
Since $g_{\mathcal C}(\cdot) >g_n (-C_n -M_n, \cdot)$, $\mathcal C(t)$ is defined at least up to time $T_0$ by Corollary \ref{C_n becomes embedded and defined up to T_0}. We write 
\begin{equation}
\mathcal C(t) = \{ x = g_{\mathcal C} (t,y) : y_-(t) \le y\le y_+(t)\}^{re},
\end{equation}
where $g_{\mathcal C} (-C_n-M_n,\cdot)=g_{\mathcal C} (\cdot)$ and $g_{\mathcal C} (t,y)$ satisfies (\ref{CSF equation by graph}).

To show that $\mathcal C(t)$ is close to $\mathcal C_{n-1}(t)$ for some $t$, we need only to show (2) and (5) in Definition \ref{closedness definition}: (1) is true for all $t$, (3) follows from Lemma \ref{g_m ge g_n} and Lemma \ref{area between g_n and g_m}, while (4) can be proved similarly as in the proof of Lemma \ref{g_m < G^-}.

By (3) and (4) we have 
$$ g_n (-C_n-M_n, y) \le g_{\mathcal C} (y) \le G^-_m (y) + a_mM_n.$$
Arguing as in the proof of Lemma \ref{g_m ge g_n} and Lemma \ref{g_m < G^-}, we have
\begin{equation} \label{g_n <g_C< G^-_n}
 g_n (t, y) \le g_{\mathcal C} (t, y) \le G^-_m (y) + a_m(t+C_m)
 \end{equation}
for all $-C_n - M_n< t < -C_n - 1$ and $m=1,2, \cdots, n$. Just like $g_n(t, y)$, for all $t < -C_n-1$, $g_{\mathcal C} (t, y)$ intersects $\{x=0\}$ at 
$$ y^-_{\mathcal C, 0} (t) < y^+_{\mathcal C, 0} (t)< \cdots < y^-_{\mathcal C, n-1} (t) < y^+_{\mathcal C, n-1} (t)< y^-_{\mathcal C, n}(t) <y^+_{\mathcal C, n}(t).$$

Using Lemma \ref{area between g_n and G^-_n} and Lemma \ref{C^0 control of C} (with suitable parabolic rescalings the reflection $R$), there are $E^-_n, \underline M^-_n >0$ depending only on $\{a_n\}$ so that
$$g_n (-C_n-1, y) +a_n \ge -E^-_n, \ \ \forall y\in (y^+_{n-1, n}(-C_n-1) , y^-_{n,n}(-C_n-1))$$ 
whenever $M_n \ge \underline M^-_n$. By (\ref{g_n <g_C< G^-_n}), we have 
\begin{equation} 
g_{\mathcal C} (-C_n-1, y)\ge -E^-_n- a_n, \ \ \forall y\in (y^+_{\mathcal C, n-1}(-C_n-1), y^-_{\mathcal C, n}(-C_n-1)) .
\end{equation}
Similarly, using Lemma \ref{area bound between g_n and o_n} and Lemma \ref{C^0 control of C} (with suitable parabolic rescalings and $A = L \sum a_n^{-2} + A_n$), there are $E^+_n, \underline M^+_n >0$ such that if $M_n \ge \underline M^+_n$, then 
\begin{equation}
g_{\mathcal C} (-C_n-1, y) - a_n(1+ La_n^{-2}) \le E^+_n, \ \ \forall y \ge y^-_{\mathcal C, n}(-C_n-1). 
\end{equation}
Choose
$$\underline M_n = \max\{ \underline M^-_n, \underline M^+_n\}, \ \ E_n = \max\{ E^-_n + a_n, E^+_n +a_n + La_n^{-1}\}.$$
Then whenever $M_n \ge \underline M_n$, $|x| \le E_n$ for all $(x, y)\in \mathcal C(-C_n-1)$ and $y\ge y^+_{\mathcal C, n-1} (-C_n-1)$. 

Next we construct another barrier $\mathcal P_{n}$. 
At time $t=-C_n-1$, $\mathcal G^-_n(-C_n-1)$ intersect $x= E_n$ at two points $p_n< p^n$. Define
\begin{equation*}
\mathcal P_n = X^{re}, 
\end{equation*}
where
\begin{align*}
X = \{ \mathcal G^-_{n} (-C_n-1) : y \le p_n\} \cup \{  E_n\}\times [ p_{n} , 2\pi h_{n} + \pi a_{n+1}^{-1}]\cup [0, E_n] \times \{ 2\pi h_{n} + \pi a_{n+1}^{-1}\}.
 \end{align*}
Let $\mathcal P_{n}(t)$ be the CSF with $\mathcal P_{n} (0) = \mathcal P_n$. For all $t > 0$, $\mathcal P_{n} (t)$ is graphical in $y$ when $x>0$, and $\mathcal P_n(t)$ is given by 
\begin{equation}
 \mathcal P_{n}(t) = \{ x = p_{n} (t,y) : 2\pi h_{n-1} < y \le p^{n}(t)\}^{re}.
\end{equation}
Also $p_{n} (t, \cdot)> g_{\mathcal C}( -C_n-1+t,\cdot)$ for all $t >0$ by maximum principle. Note that $\mathcal P_n (t)$ is an entire graph in $x$ for $t >0$. 
 
From \cite{Po}, \cite{H}, the CSF $\mathcal P_n(t)$ tends to $\{ y = 2\pi h_{n-1}\}$ uniformly as $t\to \infty$. Thus there is $T^1_{\mathcal P_n}>0$ so that $\mathcal P_n(t)$ lies between $\{ y = 2\pi h_{n-1}\}$ and $\{ y = 2\pi h_{n-1} + \pi a_n^{-1}\}$ for all $t > T_{\mathcal P_n}$. The maximum principle implies that $\mathcal C(t)$ also lies in $\{ y < 2\pi h_{n-1} +\pi a_n^{-1}\}$ for all $t > -C_n-1 + T_{\mathcal P_n}$. 

On the other hand, we have 

\noindent {\bf Claim:} There is $T^2_n>0$ independent of $B_n$ so that $\mathcal C(-C_n-1 + T^2_n)$ has exactly $n$ local maximum and $n-1$ local minimum. 

\begin{proof}[Proof of Claim:] Let $\mathcal G(t)$ be a Grim Reaper which
\begin{enumerate}
\item [(i)] is symmetric about the $y$-axis with width larger than $2E_n$, 
\item [(ii)] moves along the negative $y$ direction, and
\item [(iii)] at $t=0$, $\mathcal G(0)$ intersects $\mathcal C(-C_n-1)$ at two points $(\pm x_0, y_0)$ in the region $\{ y\ge 2\pi h_{n-1}\}$ with $|x_0|>E_n$.
\end{enumerate} 
Note $\mathcal G(t)$ can be chosen using only $E_n, h_{n}, a_n$ and is independent of the curve $\mathcal C(-C_n-1)$. Since $\mathcal G(t)$ moves at constant speed in the negative $y$-direction, there is $T^2_n>0$ so that $\mathcal G (T^2_n)$ do not intersect the region $\{y \ge 2\pi h_{n-1}\}$. We argue that the function $g_{\mathcal C}(-C_n-1+T^2_n, y)$ has no local maximum and local minimum in the region $\{y\ge 2\pi h_{n-1}\}$, except at the boundary. 

We argue by contradiction. If $g_{\mathcal C}( -C_n-1+T^2_n, \cdot)$ does has a local minimum in $\{ y\ge 2\pi h_{n-1}\}$, then the same is true for the function $g_{\mathcal C}(-C_n-1+t, \cdot)$ for all $t \in [0, T^2_n]$ by maximum principle. For any $t$, let the local maximum and local minimum of $g_{\mathcal C}(-C_n-1+t, \cdot)$ be occurred at $a(t), b(t)$ respectively, with $ y^+_n(-C_n-1+t) > a(t) >b(t)$. Let $e(t)\in \mathbb R$ so that $\mathcal G(t)$ intersects $\mathcal C(-C_n-1+t)$ at $(\pm x_t, e(t))$, and let $t_{max}$ be the time when $\mathcal G(t_{max})$ and $\mathcal C(-C_n-1 + t_{max})$ intersect at only one point (by symmetry the point of intersection is $(0,y^n(-C_n-1+ t_{max})$). 

Note that when $t$ is small, we have $b(t) > e(t)$ and at $t = t_{max}$, 
$$e(t_{max}) = y^n(-C_n-1+ t_{max}) > b(t_{max}).$$ 
By continuity, there is a time $t \in (0,t_{max})$ so that $e(t) = b(t)$. But this is impossible: since $\mathcal G(t)$ is convex, if $\mathcal G(t)$ and $\mathcal C(-C_n-1+t)$ does intersect at the local minimum, it must also intersect at yet another point. This is impossible since the number of intersections along two CSF is non-increasing by maximum principle. This finishes the proof of the claim. 
\end{proof}
  
With the above Claim we finish the proof of Proposition. Let $\overline M_n= \max\{ T_{\mathcal P_n}, T^2_n\}$ and assume that $B_n \ge \max\{ \overline M_n, B^e_n\}$. Then for all $t\in (\overline M_n, B_n)$, $-C_n-1 + t < -C_n -1 + B_n = -C_{n-1}-1$. By the choice of $\overline M_n$, we see that whenever $t\ge \overline M_n$, (2) and (5) are also satisfied for $\mathcal C(-C_n-1 +t)$, and thus $\mathcal C(-C_n-1 +t)$ is close to $\mathcal C_{n-1} (-C_{n-1}-1 + t$). 
\end{proof}

\section{Proof of Theorem \ref{Compact example}}
In this section, we prove Theorem \ref{Compact example} \label{Proof of Thm}

\begin{thm} \label{epsilon closedness for all n}
One can choose $B_n>B^e_n$, $N_n>0$ and $t_n\in (-C_{n+1}, -C_n-1)$ so that for all $m\ge n$, $\mathcal C_m(t_n)$ is close to $\mathcal C_n(t_n)$ and the curvatures of $\mathcal C_m(t_n)$ are uniformly bounded by $N_n$.  
\end{thm}

\begin{proof}
Let $\underline M_n$ and $\overline M_n$ be chosen as in Proposition \ref{epsilon closedness - inductive prop}. Note that the choice of $\overline M_n$, $\underline M_n$ depends only on the sequence $\{ a_n\}$.

Now we choose $\{ B_n\}$ so that for all $n\in \mathbb N$, $B_n \ge \{ B^e_n, \underline M_{n-1} + \overline M_n\}$. Then we have 
\begin{equation} \label{final choice of B_n}
-C_n + \overline M_n < -C_{n-1} - \underline M_{n-1},  \ \ \forall n\in \mathbb N.
\end{equation}
Let $t_{n} = -C_{n+1} +\overline M_{n+1}$. Let $m, n\in \mathbb N$ with $m >n$. By definition, $\mathcal C_m (t_m)$ is close to $\mathcal C_m(t_m)$. Using (\ref{final choice of B_n}) and proposition \ref{epsilon closedness - inductive prop}, $\mathcal C_m(t_{m-1})$ is close to $\mathcal C_{m-1} (t_{m-1})$. After finitely many steps we see that $\mathcal C_m(t_n)$ is close to $\mathcal C_n(t_n)$. 

Lastly, we show that there are $N_1, \cdots, N_n, \cdots$ so that the curvatures of $\mathcal C_m (t_n)$ are bounded by $N_n$ for all $m\ge n$. Since $\mathcal C_m (t_n)$ is close to $\mathcal C_n (t_n)$, the function $g_m (t_n, y)$ has $n+1$ local maximum and $n$ local minimum. Recall $g_m(t_n, y)$ intersects the $y$-axis at 
$$ y^-_{0,m} (t_n) < y^+_{0,m}(t_n) < \cdots < y^-_{n,m}(t_n) < y^+_{n,m} (t_n). $$
For each $k=0, \cdots, n$ and $t \le t_n$, let $\mathcal T_k (t) $ be the $k$-th tip region 
$$ \mathcal T_k (t):= \{ (x, y) \in \mathcal C_m(t) : y\in (y^+_{k-1, m}(t), y^+_{k,m}(t)), |x| > -a_k t_n -1\}.$$  
Note that in the tip regions, one can apply Lemma \ref{C^0 control of C} to obtain uniform $C^0$ distance estimates to the Grim Reapers. The gradient estimates \cite[Corollary 5.3]{ESIII} implies that $|g_m'(t, y)| \le C(n)$ in the tip regions. Since $g_m$ satisfies the graphical CSF equation (\ref{CSF equation by graph}), standard parabolic estimates implies uniform control on $g_m''(t_n, y)$ in the tip regions, which gives a uniform bound for the curvature there. 

Away from the tip regions, $\mathcal C_m(t)$ can be written as a union of $2n$-graphs in $x$ with uniform $C^0$ bounds, at least when $t\in [t_n-1, t_n]$. Then one can similarly obtain gradient estimates and thus uniform bounds on curvatures. 
\end{proof}

\begin{proof}[Proof of Theorem \ref{Compact example}] From Theorem \ref{epsilon closedness for all n}, one can find a sequence of bounded open sets $U_n$ so that $\overline{U_{n}} \subset U_{n+1}$, and $\mathcal C_m (t_n) \subset U_n$ for all $m\ge n$. The curvatures of $ \mathcal C_m (t_n)$ are uniformly bounded for $m\ge n$. By choosing a slightly larger time $t_n$ for each $n$ and using the CSF equation, one can assume that all derivatives of curvatures of $\mathcal C_m(t_n)$ are uniformly bounded for $m\ge n$ \cite[Section 3]{BL}. Thus by \cite[Theorem 1.3]{Breuning}, we can pick a diagonal subsequence $\{n_j\}_{j=1}^\infty$ so that for all $n\in \mathbb N$, $\{ \mathcal C_{n_j} (t_n)\}_{j=1}^\infty$ converges smoothly to a compact immersion $\mathcal C^n$. Let $\mathcal C^n(t)$ be the CSF with $\mathcal C^n (t_n) = \mathcal C^n$. Since $\mathcal C_m (t_n)$ converges smoothly to $\mathcal C^n$, the continuous dependence of CSF implies that the sequence of CSF's $\{\mathcal C_m(t) : t_n \le t\le T_0\}$ converges smoothly to $\{\mathcal C^n(t) : t_n \le t\le T_0\}$. In particular, for all $m >n$,  
$$ \mathcal C^m (t) = \mathcal C^n (t), \ \ \ \forall t \ge t_n.$$
Thus we can define, for each $t <T_0$, $\mathcal C(t) = \mathcal C^m (t)$, where $m\in \mathbb N$ and $t_m < t$. Then $\{\mathcal C (t) :  t< T_0\}$ is a CSF so that $\mathcal C(t_n) = \mathcal C^n$ for all $n\in \mathbb N$. Since $\mathcal C (t_n)$ is the limit of $\mathcal C_m(t_n)$ as $m\to \infty$,
$$ \bigcup _{t<T_0} \operatorname{Conv} \mathcal C(t) = \{ (x, y) : y > -\pi\}$$
is a halfspace. Lastly, in the construction we fix a choice of sequence $\{a_n\}$. Since different sequences $\{a_n\}$ lead to ancient solutions with different limit as $t\to -\infty$, the above construction gives an infinite family of ancient solutions. 
\end{proof}

\bibliographystyle{amsplain}

\end{document}